\documentclass[a4paper,oneside,10pt]{article}

\usepackage[english]{babel}
\usepackage[latin1]{inputenc}
\usepackage[T1]{fontenc}
\usepackage{lmodern}
\usepackage{cite}
\usepackage{microtype}
\usepackage{tikz}
\usetikzlibrary{patterns}
\usepackage{perpage}
\MakePerPage{footnote}

\usepackage{amssymb}
\usepackage{amsmath}
\usepackage{mathrsfs}
\usepackage{dsfont}
\usepackage{exscale}
\usepackage{relsize}
\usepackage{amsthm}

\newcommand{\ind}[1]{\mathds{1}_{#1}}

\newcommand{\R}{\mathbb{R}}
\renewcommand{\P}{\mathbb{P}}
\newcommand{\E}{\mathbb{E}}

\newcommand{\Ll}{\mathrm{L}}

\newcommand{\Ac}{\mathcal{A}}
\newcommand{\Bc}{\mathcal{B}}

\newcommand{\Dc}{\mathcal{D}}

\newcommand{\Fc}{\mathcal{F}}
\newcommand{\Gc}{\mathcal{G}}

\newcommand{\Mc}{\mathcal{M}}
\newcommand{\Nc}{\mathcal{N}}

\newcommand{\Sc}{\mathcal{S}}
\newcommand{\Tc}{\mathcal{T}}
\newcommand{\Uc}{\mathcal{U}}

\renewcommand{\a}{\alpha}
\renewcommand{\b}{\beta}

\newcommand{\G}{\Gamma}
\renewcommand{\d}{\delta}

\renewcommand{\epsilon}{\varepsilon}
\newcommand{\eps}{\varepsilon}
\newcommand{\z}{\zeta}
\renewcommand{\r}{\rho}
\newcommand{\s}{\sigma}

\renewcommand{\t}{\tau}

\renewcommand{\phi}{\varphi}
\renewcommand{\k}{\kappa}

\newcommand{\Ck}[1]{C^{#1}}
\newcommand{\loi}{\mathscr{L}}
\newcommand{\gfrac}[2]{\genfrac{}{}{0pt}{1}{#1}{#2}}

\def\Ccro{\smash{{\mathcal{C}}^{\!\!\!\raise4pt\hbox{$\scriptstyle o$}}}}
\def\Aro{\smash{{A}^{\!\!\!\raise5pt\hbox{$\scriptstyle o$}}}}
\def\Aro2{\smash{{A}^{\!\!\!\raise4pt\hbox{$\scriptstyle o$}}}}

\newcommand{\Ss}{\smash{\widetilde{S}}}
\newcommand{\Ts}{\smash{\widetilde{T}}}
\newcommand{\Gs}{\smash{\widetilde{G}}}

\newtheorem{theo}{Theorem}
\newtheorem{defi}[theo]{Definition}

\newtheorem{prop}[theo]{Proposition}

\newtheorem{lem}[theo]{Lemma}

\renewenvironment{proof}{\noindent{\bf Proof.}}{\qed}
\renewcommand{\thesection}{\arabic{section}}

\begin{document}

\renewcommand{\contentsname}{Contents}
\renewcommand{\refname}{\textbf{References}}
\renewcommand{\abstractname}{Abstract}

%\shorthandoff{:}
%\noindent A revised version of this paper will appear soon in\\ \textbf{\textit{Annales de l'Institut Henri Poincar\'e}.}\bigskip \bigskip \bigskip

\begin{center}
%\begin{Huge}
%A Gaussian Dynamical \medskip
%
%Curie-Weiss Model of SOC
%\end{Huge} \bigskip \bigskip \bigskip \bigskip

\begin{Huge}
A Dynamical Curie-Weiss Model\medskip

of SOC: The Gaussian Case
\end{Huge} \bigskip \bigskip \bigskip \bigskip

\begin{Large} Matthias Gorny \end{Large} \smallskip
 
\begin{large} {\it Universit\'e Paris-Sud and ENS Paris} \end{large} \bigskip \bigskip 
\end{center}
\bigskip \bigskip \bigskip

\begin{abstract}
\noindent In this paper, we introduce a Markov process whose unique invariant distribution is the Curie-Weiss model of self-organized criticality (SOC) we designed and studied in~\cite{CerfGorny}. In the Gaussian case, we prove rigorously that it is a dynamical model of SOC: the fluctuations of the sum $S_{n}(\,\cdot\,)$ of the process evolve in a time scale of order $\sqrt{n}$ and in a space scale of order $n^{3/4}$ and the limiting process is the solution of a "critical" stochastic differential equation.
\end{abstract}
\bigskip \bigskip \bigskip \bigskip \bigskip



\newpage

\section{Introduction}

In~\cite{CerfGorny} and~\cite{GORGaussCase}, we introduced a Curie-Weiss model of self-organized criticality (SOC): we transformed the distribution associated to the generalized Ising Curie-Weiss model by implementing an automatic control of the inverse temperature which forces the model to evolve towards a critical state. It is the model given by an infinite triangular array of real-valued random variables $(X_{n}^{k})_{1\leq k \leq n}$ such that, for all $n \geq 1$, $(X^{1}_{n},\dots,X^{n}_{n})$ has the distribution
\[\frac{1}{Z_{n}}\exp\left(\frac{1}{2}\frac{(x_{1}+\dots+x_{n})^{2}}{x_{1}^{2}+\dots+x_{n}^{2}}\right)\ind{\{x_{1}^{2}+\dots+x_{n}^{2}>0\}}\,\prod_{i=1}^{n}d\r(x_{i}),\]
where $\r$ is a probability measure on $\R$ which is not the Dirac mass at~$0$, and where $Z_n$ is the normalization constant. We extended the study of this model in~\cite{Gorny3},~\cite{GornyThesis} and~\cite{GorVar}. For symmetric distributions satisfying some exponential moment condition, we proved that the sum $S_{n}$ of the random variables behaves as in the typical critical generalized Ising Curie-Weiss model: the fluctuations are of order $n^{3/4}$ and the limiting law is $C \exp(-\lambda x^{4})\,dx$ where $C$ and $\lambda$ are suitable positive constants. Moreover, by construction, the model does not depend on any external parameter. That is why we can conclude it exhibits the phenomenon of self-organized criticality (SOC). Our motivations for studying such a model are detailed in~\cite{CerfGorny}.\medskip

This model describes interacting elements in thermodynamic equilibrium. However self-organized criticality seems to be a dynamical phenomenon, as is highlighted by the archetype of SOC : the sandpile model introduced by Per Bak, Chao Tang and Kurt Wiesenfeld in their seminal 1987 paper~\cite{BTW1f}. That is why, in this paper, we try to design a dynamical Curie-Weiss model of SOC.\medskip

We choose to build a dynamical model as a Markov process whose unique invariant distribution is the law of (a modified version of) the Curie-Weiss model of SOC. One way  of building such a process is to consider the associated Langevin diffusion (see~\cite{RobTweed} for example).\medskip

\textbf{The model.} Let $\phi$ be a $\Ck{2}$ function from $\R$ to $\R$ which is even and such that the function $\exp(2\phi)$ is integrable over $\R$. We suppose that there exists $C>0$ such that
\[\forall x\in \R\qquad x\phi'(x)\leq C(1+x^2).\]
We denote by $\r$ the probability measure with density
\[x\longmapsto \exp(2\phi(x))\,\left(\int_{\R}\exp(2\phi(t))\,dt\right)^{-1}\]
with respect to the Lebesgue measure on $\R$. We consider an infinite triangular array of stochastic processes $(X_{n}^{k}(t),\,t\geq 0)_{1\leq k \leq n}$ such that, for all $n \geq 1$,
\[\smash{\big((X^{1}_{n}(t),\dots,X^{n}_{n}(t)),\,t\geq 0\big)}\]
is the unique solution of the system of stochastic differential equations:
\begin{equation}
\begin{split}
dX_n^j(t)=\phi'\big(X_n^j&(t)\big)\,dt+dB_j(t)\\
&\begin{split}+\frac{1}{2}\bigg(\frac{S_n(t)}{T_n(t)+1}-X_n^j(t)\bigg(\frac{S_n(t)}{T_n(t)+1}\bigg)^2&\bigg)\,dt\\
&j\in \{1,\dots,n\},\end{split}
\end{split}
\tag{$\mathbf{\Sigma}_n^{\phi}$}
\end{equation}
%\begin{multline*}
%dX_n^j(t)=\phi'\big(X_n^j(t)\big)\,dt+dB_j(t)+\frac{1}{2}\left(\frac{S_n(t)}{T_n(t)+1}-X_n^j(t)\left(\frac{S_n(t)}{T_n(t)+1}\right)^2\right)\,dt\\
%j\in \{1,\dots,n\},
%\end{multline*}
where $(B_1,\dots,B_n)$ is a standard $n$-dimensional Brownian motion and
\[\forall t\geq 0 \qquad S_n(t)=X_n^1(t)+\dots+X_n^n(t),\qquad T_n(t)=\big(X_n^1(t)\big)^2+\dots+\big(X_n^n(t)\big)^2.\]

In section~\ref{generalresults}.\ref{solutions}, we explain in details why we choose this drift. In this paper, we only prove a fluctuation theorem for the Gaussian case of this model:

\begin{theo} Let $\s^2>0$. Assume that
\[\forall x\in \R\qquad \phi(x)=-\frac{x^2}{4\s^2}\]
and that, for any $n\geq 1$, the random variables $X_n^1(0),\dots,X_n^n(0)$ are independent with common distribution $\r=\Nc(0,\s^2)$. We denote $(\Uc(t),\,t\geq 0)$ the unique strong solution of the stochastic differential equation
\begin{equation}
dz(t)=-\frac{z^3(t)}{2\s^4}\,dt+dB(t),\qquad z(0)=0,
\tag{$\mathcal{S}_{\s}$}
\end{equation}
where $(B(t),\,t\geq 0)$ is a standard Brownian motion. Then, for any $T>0$,
\[\displaystyle{\left(\frac{S_n(\sqrt{n}t)}{n^{3/4}},\,0\leq t\leq T\right)\overset{\loi}{\underset{n \to +\infty}{\longrightarrow}}\big(\Uc(t),\,0\leq t\leq T\big)},\]
in the sense of the convergence in distribution on $C([0,T],\R)$.
\label{MainTheoGauss}
\end{theo}

This theorem suggests that, at least in the Gaussian case, our dynamical model exhibits self-organized criticality. Indeed it does not depend on any external parameter and the fluctuations of $S_{n}(\,\cdot\,)$ are critical: the processes evolve in a time scale of order $\sqrt{n}$ and in a space scale of $n^{3/4}$ and the limiting process is the solution of the "critical" stochastic differential equation $(\mathcal{S}_{\s})$. This is the same behaviour as in the critical case of the mean-field model studied by Donald~A.~Dawson in~\cite{Dawson}, see section~\ref{strategy}.\ref{empirical} for more details.\medskip

For any $n\geq 1$, we introduce $S_n^{\star}=\xi_n^1+\dots+\xi_n^n$ where $(\xi_n^1,\dots,\xi_n^n)$ has the density proportional to
\[(x_{1},\dots,x_{n})\longmapsto\exp\left(\frac{1}{2}\frac{(x_{1}+\dots+x_{n})^{2}}{x_{1}^{2}+\dots+x_{n}^{2}+1}-\frac{x_{1}^{2}+\dots+x_{n}^{2}}{2\s^2}\right)\]
with respect to the Lebesgue measure on $\R^n$. In this paper, we also prove the following commutative diagram of convergences in distribution on $\R$:
\begin{center}
\begin{tikzpicture}
\draw (0,0) node {$\displaystyle{\frac{S_n(\sqrt{n} t )}{n^{3/4}}}$} ;
\draw [thick] [>=stealth,->] (0,-1) -- (0,-3) ;
\draw (0,-4) node {$\Uc(t)$} ;
\draw [thick] [>=stealth,->] (2.5,0) -- (5,0) ;
\draw (8.5,0) node {$\displaystyle{\frac{S_n^{\star}}{n^{3/4}}}$} ;
\draw [thick] [>=stealth,->] (8.5,-1) -- (8.5,-3) ;
\draw (8.5,-4) node {$\displaystyle{\frac{\sqrt{2}}{\s}\,\G\left(\frac{1}{4}\right)^{-1} \exp\left(-\frac{s^4}{4\s^{4}}\right)\,ds}$} ;
\draw [thick] [>=stealth,->] (2.5,-4) -- (5,-4) ;
\draw (0.25,-2) node [rotate=270]  {$n\to+\infty$} ;
\draw [color=gray](-0.4,-2) node  {$(\Ac_4)$} ;
\draw (8.25,-2) node [rotate=270]  {$n\to+\infty$} ;
\draw [color=gray](8.9,-2) node {$(\Ac_2)$} ;
\draw [color=gray](3.75,0.35) node  {$(\Ac_1)$} ;
\draw (3.75,-0.25) node  {$t\to+\infty$} ;
\draw [color=gray](3.75,-4.35) node  {$(\Ac_3)$} ;
\draw (3.75,-3.75) node  {$t\to+\infty$} ;
\label{Diagram}
\end{tikzpicture}
\end{center} \medskip

In section~\ref{generalresults}, we present some results on the general case of the model and we prove the convergences in distribution associated to the arrows $(\Ac_1)$, $(\Ac_2)$ and $(\Ac_3)$ in the previous diagram. Next, in section~\ref{strategy}, we give the strategy for proving a fluctuation result for our model and we explain that the Gaussian case is special because it can be analyzed through a two-dimensional problem. Finally we prove theorem~\ref{MainTheoGauss} in section~\ref{proof}, i.e., the convergence in distribution associated to the arrow~$(\Ac_4)$.\medskip

\textbf{Acknowledgments.} The author would like to thank the anonymous referee for his careful review and useful comments which helped to improve the presentation of the paper.

\section{Results on the general case of the model}
\label{generalresults}

In this section, we first give general results on Langevin diffusions. Next we apply these results to prove existence and uniqueness of the solution of $(\Sc_{\s})$ and $(\mathbf{\Sigma}_n^{\phi})$. We also prove the convergences in distribution associated to the arrows $(\Ac_1)$ and $(\Ac_3)$. Finally we give a fluctuation theorem for an alternative version of the Curie-Weiss model of SOC.

\subsection{Langevin diffusions}

Let $f$ be a probability density function on $\R^n$, $n\geq 1$. The Langevin diffusion associated to $f$ is a stochastic process which is constructed so that, in continuous time, under suitable regularity conditions, it converges to $f(x)\,dx$, its unique invariant distribution.

\begin{theo} Let $f$ be a positive probability density function on $\R^n$, $n\geq 1$, such that $\ln\!f$ is $\Ck{2}$. We suppose that there exists $K>0$ such that
\[\forall x\in \R^n\qquad \langle\nabla \ln\!f(x),x\rangle\leq K(1+\|x\|^2).\]
If $(B(t),\,t\geq 0)$ is a standard $n$-dimensional Brownian motion and if $\xi$ is a random variable in $\R^n$ satisfying $\smash{\E\big(\|\xi\|^2\big)<+\infty}$, then there exists a unique strong solution to the stochastic differential equation
\begin{equation}
dY(t)=\frac{1}{2}\nabla\ln\!f(Y(t))+dB(t),
\tag{$\Sc_f$}
\end{equation}
with initial condition $Y(0)=\xi$. Moreover $(Y(t),\,t\geq 0)$ is a Markov diffusion process on $\R^n$ admitting $f(x)\,dx$ as unique invariant distribution and
\[\forall x\in \R^n\qquad \lim_{t\to+\infty}\,\sup_{A\in \Bc_{\R^n}}\,\left|\,\P\big(Y(t)\in A \,\big|\,Y(0)=x\big)-\int_Af(z)\,dz\,\right|=0.\]
\label{theoLangevin}
\end{theo}

\begin{proof} Theorems~3.7 and~3.11 of chapter~5 of~\cite{EKurtz} imply that there exists a unique strong solution to $(\Sc_f)$ with initial condition $\xi$, that its sample path is continuous and that it is a solution of the martingale problem for $(A_f,\xi)$, where
\[\forall g\in C^2(\R^n)\qquad A_fg=\frac{1}{2}\sum_{i=1}^n\frac{\partial^2 g}{\partial x_i^2}+\sum_{i=1}^n\left(\frac{1}{2}\frac{\partial (\ln\!f)}{\partial x_i}\right)\frac{\partial g}{\partial x_i}.\]
Next, theorems~4.1 and~4.2 of chapter~4 of~\cite{EKurtz} imply that it is a Markov process and that its generator is $(A_f,D(A_f))$ with $C_c^{\infty}(\R^n)\subset D(A_f)$. Finally theorem~2.1~of~\cite{RobTweed} gives us the uniqueness of the invariant distribution and the total variation norm convergence.\end{proof}\medskip

Notice that this theorem is true if we remove the hypothesis that $\xi$ has a finite second order moment, but the solution to $(\Sc_f)$ would be weak (see theorem~3.10 of chapter~5 of~\cite{EKurtz}).

\subsection{Solution of $(\Sc_{\s})$}

Theorem~\ref{theoLangevin} implies that $(\Sc_{\s})$ admits a unique strong solution $(\Uc(t),\,t\geq 0)$ which is a Markov process whose unique invariant distribution is
\[\frac{\sqrt{2}}{\s}\,\G\left(\frac{1}{4}\right)^{-1} \exp\left(-\frac{s^4}{4\s^{4}}\right)\,ds.\]
Moreover
\[\lim_{t\to+\infty}\,\sup_{A\in \Bc_{\R}}\,\left|\,\P\big(\Uc(t)\in A \big)-\frac{\sqrt{2}}{\s}\,\G\left(\frac{1}{4}\right)^{-1}\int_{A}\exp\left(-\frac{s^4}{4\s^{4}}\right)\,ds\,\right|=0.\]
This is the convergence in distribution associated to the arrow $(\Ac_3)$ in the diagram on page~\pageref{Diagram}.

\subsection{Solution of $(\mathbf{\Sigma}_n^{\phi})$}
\label{solutions}

%\newcommand{\EnsPhi}{\Xc}
%
%We denote by $\EnsPhi$ the set of $\Ck{2}$ functions $\phi :\R\longrightarrow\R$ such that $\phi$ is even, $\exp(2\phi)$ is integrable over $\R$ and there exists $C>0$ such that
%\[\forall x\in \R\qquad x\phi'(x)\leq C(1+x^2).\]
%Let $\phi \in \EnsPhi$. We denote by $\r$ the probability measure having a density proportional to $\exp(2\phi)$ with respect to the Lebesgue measure on $\R$.
In this subsection, we prove that $(\mathbf{\Sigma}_n^{\phi})$ has a unique strong solution and that the convergence in distribution associated to $(\Ac_1)$ is true.\medskip

Let us define $\widetilde{\mu}^{\star}_{n,\r}$, the probability measure with density
\begin{equation}
f^{\star}_{n,\r}:y\in \R^n\longmapsto\frac{1}{Z_n^{\star}}\exp\left(\frac{1}{2}\frac{\left(y_1+\dots+y_n\right)^2}{y_1^2+\dots+y_n^2+1}+2\sum_{i=1}^n\phi(y_i)\right)
\label{altCWSOC}
\end{equation}
with respect to the Lebesgue measure on $\R^n$, where $Z_n^{\star}$ is a normalization constant. Let us prove that $(\mathbf{\Sigma}_n^{\phi})$ admits a unique solution. For any $y\in \R^n$, we denote
\[S_n[y]=y_1+\dots+y_n,\qquad T_n[y]=y_1^2+\dots+y_n^2\]
and we notice that, for any $j\in \{1,\dots,n\}$,
\[\frac{\partial}{\partial y_j}\left(\frac{1}{2}\frac{\left(S_n[y]\right)^2}{T_n[y]+1}+2\sum_{i=1}^n\phi(y_i)\right)=\frac{S_n[y]}{T_n[y]+1}-y_j\left(\frac{S_n[y]}{T_n[y]+1}\right)^2+2\phi'(y_j).\]
Therefore the system $(\mathbf{\Sigma}_n^{\phi})$~can be rewritten
\[dX_n(t)=\frac{1}{2}\nabla \ln\!f^{\star}_{n,\r}(X_n(t))+dB(t),\]
where $B=(B_1,\dots,B_n)$. As a consequence, the solution of $(\mathbf{\Sigma}_n^{\phi})$ (if it exists) is the Langevin diffusion associated to $f^{\star}_{n,\r}$.\medskip

Let us introduce the operator $L_n$ on $C^2(\R^n)$ such that, for any $f\in C^2(\R^n)$ and $y \in \R^n$,
\[L_nf(y)=\frac{1}{2}\sum_{j=1}^n\frac{\partial^2 f(y)}{\partial y_j^2}+\sum_{j=1}^n\left(\frac{1}{2}\frac{S_n[y]}{T_n[y]+1}-\frac{y_j}{2}\left(\frac{S_n[y]}{T_n[y]+1}\right)^2\!\!+\phi'(y_j)\right)\frac{\partial f(y)}{\partial y_j}.\]

\begin{theo} For any $n\geq 1$, there exists a unique strong solution
\[\big(X_n(t),\,t\geq 0\big)=\big((X_n^1(t),\dots,X_n^n(t)),\,t\geq 0\big)\] to the system $(\mathbf{\Sigma}_n^{\phi})$~with initial condition $X_n(0)$ having a finite second moment. Moreover it is a Markov diffusion process on~$\R^n$ with infinitesimal generator $(L_n,D(L_n))$, where $C^{\infty}_c(\R^n)\subset D(L_n)$, and whose unique invariant distribution is $\widetilde{\mu}^{\star}_{n,\r}$. Finally
\[\forall x\in \R^n\qquad \lim_{t\to+\infty}\,\sup_{A\in \Bc_{\R^n}}\,\left|\,\P\big(X_n(t)\in A \,\big|\,X_n(0)=x\big)-\widetilde{\mu}^{\star}_{n,\r}(A)\,\right|=0.\]
\label{ExistenceSys}
\end{theo}

If we take $\phi(x)=-x^2/(4\s^2)$ for any $x\in \R$, then theorem~\ref{ExistenceSys} proves the convergence in distribution associated to the arrow $(\Ac_1)$ in the diagram on page~\pageref{Diagram}.\medskip

\begin{proof} Let $n\geq 1$. By hypothesis, there exists $C>0$ such that
\[\forall x\in \R\qquad x\phi'(x)\leq C(1+x^2).\]
Moreover $\phi$ is $\Ck{2}$ on $\R$ thus the function $\ln\!f^{\star}_{n,\r}$ is $\Ck{2}$ on $\R^n$. For any $x\in \R^n$, we have
\begin{align*}
\langle \nabla \ln\!f^{\star}_{n,\r}(x),x\rangle&=\sum_{j=1}^n x_j\, \frac{\partial}{\partial x_j}\left(\frac{1}{2}\frac{\left(S_n[x]\right)^2}{T_n[x]+1}+2\sum_{i=1}^n\phi(x_i)\right)\\
&=\sum_{j=1}^n x_j\left(\frac{S_n[x]}{T_n[x]+1}-x_j\left(\frac{S_n[x]}{T_n[x]+1}\right)^2+2\phi'(x_j)\right)\\
&=\frac{(S_n[x])^2}{T_n[x]+1}-\frac{T_n[x](S_n[x])^2}{(T_n[x]+1)^2}+2\sum_{j=1}^nx_j\phi'(x_j)\\
&\leq \frac{(S_n[x])^2}{(T_n[x]+1)^2}+2C(n+\|x\|^2).
\end{align*}
Next the convexity of $t\longmapsto t^2$ on $\R$ implies that
\[\forall y\in \R^n\qquad \frac{(S_n[y])^2}{(T_n[y]+1)^2}\leq \frac{n T_n[y]}{(T_n[y]+1)^2}\leq n,\]
since $T_n[\cdot] \leq (T_n[\cdot]+1)^2$. Therefore $f^{\star}_{n,\r}$ satisfies the hypothesis of theorem~\ref{theoLangevin} and theorem~\ref{ExistenceSys} follows.
\end{proof}\medskip

Remark: we have chosen to built our dynamical model so that $\widetilde{\mu}^{\star}_{n,\r}$ is its unique invariant distribution. It is an alternative version of the Curie-Weiss model we designed in~\cite{CerfGorny}, given by the distribution
\[d\widetilde{\mu}_{n,\r}(x_1,\dots,x_n)=\frac{1}{Z_n}\exp\left(\frac{1}{2}\frac{(x_1+\dots+x_n)^2}{x_1^2+\dots+x_n^2}+2\sum_{i=1}^n\phi(x_i)\right)\,dx_1\,\cdots\,dx_n,\]
where $Z_n$ is a normalization constant. If we want to built the Langevin diffusion associated to the density of $\widetilde{\mu}_{n,\r}$, we obtain the system of stochastic differential equations
\begin{multline*}
dX_n^j(t)=\phi'\big(X_n^j(t)\big)\,dt+dB_j(t)+\frac{1}{2}\left(\frac{S_n(t)}{T_n(t)}-X_n^j(t)\left(\frac{S_n(t)}{T_n(t)}\right)^2\right)\,dt,\\
j\in \{1,\dots,n\}.
\end{multline*}
In this case, the interaction function is not Lipschitz and we have to check first that $T_n(t)\neq 0$ for any $t\geq 0$: this would create technical difficulties to prove existence and uniqueness of a solution. In the next section, we give some results on the alternative version of the Curie-Weiss model of SOC (the model defined by the probability measure $\widetilde{\mu}^{\star}_{n,\r}$ -- see formula~(\ref{altCWSOC})).

\subsection{The alternative Curie-Weiss model of SOC}

Let $\r$ be a probability measure on $\R$. We consider an infinite triangular array of real-valued random variables $(\xi_{n}^{k})_{1\leq k \leq n}$ such that for all $n \geq 1$, $(\xi^{1}_{n},\dots,\xi^{n}_{n})$ has the distribution
\begin{equation}
d\widetilde{\mu}^{\star}_{n,\r}(x_{1},\dots,x_{n})=\frac{1}{Z_{n}^{\star}}\exp\left(\frac{1}{2}\frac{(x_{1}+\dots+x_{n})^{2}}{x_{1}^{2}+\dots+x_{n}^{2}+1}\right)\prod_{i=1}^{n}d\r(x_{i}),
\label{altCWSOC2}
\end{equation}
where $Z_{n}^{\star}$ in the normalization constant. We define $S^{\star}_{n}=\xi^{1}_{n}+\dots+\xi^{n}_{n}$.\medskip

We obtain the same fluctuation theorem as in~\cite{Gorny3}. We only present the case where $\r$ has a density:

\begin{theo} Let $\r$ be a probability measure having an even density with respect to the Lebesgue measure on $\R$ and such that
\[\exists v_{0}>0 \qquad \int_{\R}e^{v_{0}z^{2}}\,d\r(z)<+\infty.\]
If $\s^2$ denotes the variance of $\r$ and $\mu_4$ its fourth moment then, under $\widetilde{\mu}^{\star}_{n,\r}$,
\[\frac{S^{\star}_n}{n^{3/4}}\overset{\mathscr{L}}{\underset{n \to \infty}{\longrightarrow}} \left(\frac{4\mu_{4}}{3\s^{8}}\right)^{1/4}\G\left(\frac{1}{4}\right)^{-1} \exp\left(-\frac{\mu_{4}s^4}{12 \s^{8}}\right)\,ds.\]
\label{TheoremCWSOCalt}
\end{theo}

The proof of this theorem is given in section~18.b)~of~\cite{GornyThesis}. It is an adaptation of the proof of theorem~1 of~\cite{Gorny3}, which consists in replacing the function~$F$ by the function $(x,y)\longmapsto x^2/(2y+2/n)$.\medskip
%Notice that, in the case where $\phi(x)=-x^2/(4\s^2)$ for any $x\in \R$, the proof of theorem~\ref{TheoremCWSOCalt} is a simple adaptation of the Gaussian case of the Curie-Weiss model of SOC, treated in~\cite{GORGaussCase}.

If we take $\phi(x)=-x^2/(4\s^2)$ for any $x\in \R$, then theorem~\ref{TheoremCWSOCalt} implies the convergence in distribution associated to the arrow $(\Ac_2)$ in the diagram on page~\pageref{Diagram}.

\section{Strategy of proof}
\label{strategy}

In this section, we first explain that the main ingredient for proving a fluctuation theorem for our dynamical model (in the case of a general function) will be the study of its associated empirical process. Next we will focus only on the Gaussian case, i.e., when $\phi:x\longmapsto-x^2/(4\s^2)$ for some $\s^2>0$. Indeed we will see that the Gaussian case can be handled by studying the convergence of the process
\[\left(\left(\frac{S_n(\sqrt{n}t)}{n^{3/4}},n^{1/4}\left(\frac{T_n(\sqrt{n}t)}{n}-\s^2\right)\right),\,t\geq 0\right).\]
We compute the generator of this process in subsection~\ref{generator}. Finally we give the sketch of proof of theorem~\ref{MainTheoGauss} in subsection~\ref{sketch}.

\subsection{The empirical process}
\label{empirical}

%Let $\phi \in \EnsPhi$. 
Let $\phi$ be such that $\mathbf{\Sigma}_n^{\phi}$ has a unique strong solution $\big((X_n^1(t),\dots,X_n^n(t)),\,t\geq 0\big)$. As in the equilibrium case (i.e., the alternative Curie-Weiss model defined in formula~(\ref{altCWSOC}) or~(\ref{altCWSOC2})), we would like to study the process $(S_n,T_n)$. However it is not Markov a priori, contrary to the empirical measure process $M_n$. It is the process taking its values on $\Mc_1(\R)$ and defined by
\[\forall t \geq 0\quad\forall A\in \Bc_{\R} \qquad M_n(t,A)=\frac{1}{n}\sum_{k=1}^n \d_{X_n^k(t)}(A)=\frac{1}{n}\sum_{k=1}^n \ind{A}\big(X_n^k(t)\big),\]
where $\smash{\big((X^{1}_{n}(t),\dots,X^{n}_{n}(t)),\,t\geq 0\big)}$ is the unique solution of~$(\mathbf{\Sigma}_n^{\phi})$. 

\begin{lem} If the distribution of $X_n(0)$ is invariant under permutation of coordinates, then $(M_n(t,\cdot),t \geq 0)$ is a Markov diffusion process on~$\Mc_1(\R)$.
\end{lem}

This lemma has a similar proof than lemma~2.3.1~of the article~\cite{Dawson} -- a paper by Donald A. Dawson about a mean-field model of cooperative behaviour. Dawson's model is defined through a Markov process which is solution of a system of stochastic differential equations. This process depends on two parameters and Dawson proves the existence of a critical curve in the space of the parameters. The critical fluctuations of the empirical measure process $M_{n}(\,\cdot\,)$ evolve in a time scale of order $\sqrt{n}$ and in a space scale of order $n^{3/4}$. We believe that our dynamical model has the same asymptotic behavior for the following reasons:\smallskip

$\star$ The invariant distribution of Dawson's process is a particular case of the law of the generalized Ising Curie-Weiss model, defined in~\cite{EN}. \smallskip

%$\star$ The alternative Curie-Weiss model of SOC is built by modifying the law of the generalized Ising Curie-Weiss model.\smallskip

$\star$ The alternative Curie-Weiss model of SOC, defined in formula~(\ref{altCWSOC}) or~(\ref{altCWSOC2}), has the same asymptotic behavior as the critical generalized Ising Curie-Weiss model (see theorem~\ref{TheoremCWSOCalt}).\smallskip 

$\star$ The invariant distribution of our dynamical model is the law of the alternative Curie-Weiss model (see theorem~\ref{ExistenceSys}).\medskip

%Moreover the limiting process is a solution of a stochastic differential equation of the same type of $(\mathcal{S}_{\s})$. 
%
%
%
%
%In order to do this, Dawson studies the fluctuations of the empirical process associated to his model.\medskip
%
%Moreover 
%
%Since we already know that the equilibrium distribution of our model has fluctuations of order $n^{3/4}$, and since our model is close to Dawson's model in the critical case, we also interested in the fluctuations of $M_n$. 

Let $n\geq 1$. As in Dawson's paper, we define the process $U_n$ by
\[\forall t \geq 0\quad\forall A\in \Bc_{\R} \qquad U_n(t,A)=n^{1/4}\bigg(M_n(\sqrt{n}t,A)-\int_A d\r(x)\bigg).\] 
It takes its values on $\Mc^{\pm}(\R)$, the space of signed measures on $\R$.\medskip

The convergence of a sequence of Markov processes can be proved through the convergence of the sequence of their generators. Let us denote by $G_n$ the infinitesimal generator of $U_n$. Let $f$ and $\Phi$ belong to $C^2(\R)$. We assume that $\Phi$ is $\r$-integrable. We have
\[\forall t\geq 0 \qquad G_nf\left(\int_{\R}\Phi(z)\,U_n(t,dz)\right)=\sqrt{n}L_nF_{f,\Phi}\big(X^{1}_{n}(t),\dots,X^{n}_{n}(t)\big)\]
where
\[F_{f,\Phi} :x\in \R^n\longmapsto f\left(n^{1/4}\left(\frac{1}{n}\sum_{k=1}^n\Phi(x_k)-\int_{\R}\Phi(z)\,d\r(z)\right)\right).\]
If $\Phi : z\longmapsto z$ then, for any $i\in \{1,\dots,n\}$ and $x\in \R^n$,
\[\frac{\partial F_{f,\Phi}}{\partial x_i}(x)=\frac{1}{n^{3/4}}F_{f',\Phi}(x)\qquad \mbox{and}\qquad \frac{\partial^2 F_{f,\Phi}}{\partial x_i^2}(x)=\frac{1}{n^{3/2}}F_{f'',\Phi}(x).\]
If $\Phi : z\longmapsto z^2$ then, for any $i\in \{1,\dots,n\}$ and $x\in \R^n$,
\[\frac{\partial F_{f,\Phi}}{\partial x_i}(x)=\frac{2x_i}{n^{3/4}}F_{f',\Phi}(x)\quad \mbox{and}\quad \frac{\partial^2 F_{f,\Phi}}{\partial x_i^2}(x)=\frac{4x_i^2}{n^{3/2}}F_{f'',\Phi}(x)+\frac{2}{n^{3/4}}F_{f',\Phi}(x).\]
In both cases, if we suppose that $\phi : z\longmapsto -z^2/(4\s^2)$, then we notice that, for any $x\in \R^n$, the term $L_nF_{f,\Phi}(x)$ only depends on $n$, $S_n[x]$ and $T_n[x]$. This suggests that, in the Gaussian case, in order to prove the convergence of the process $\big(S_n(\sqrt{n}t)/n^{3/4},\,t\geq 0\big)$, we can turn the study of $U_n$ (which is a problem in infinite dimensions) into a problem in only two dimensions. Indeed, we introduce the processes $\Ss_n$ and $\Ts_n$ defined by
\[\forall t\geq 0 \qquad \Ss_n(t)=\frac{S_n(\sqrt{n}t)}{n^{3/4}}=\int_{\R}z\,U_n(t,dz)\]
and
\[\forall t\geq 0 \qquad \Ts_n(t)=n^{1/4}\left(\frac{T_n(\sqrt{n}t)}{n}-\s^2\right)=\int_{\R}z^2\,U_n(t,dz).\]

In the rest of the paper, we suppose that $\phi(x)=-x^2/(4\s^2)$ for any $x\in \R$.

\subsection{Generator of $(\Ss_n,\Ts_n)$ in the Gaussian case}
\label{generator}

Let $n\geq 1$ and $f\in \Ck{2}(\R^2)$. Let us define $\Psi_f$ on $\R^n$ by
\[\forall x\in \R^n\qquad \Psi_f(x)=f\left(\frac{S_n[x]}{n^{3/4}},\frac{T_n[x]}{n^{3/4}}-n^{1/4}\s^2\right).\]

\begin{prop} For any $n\geq 1$ and $f \in C^2(\R^2)$, we have
\[\forall t\geq 0\qquad \sqrt{n}L_n\Psi_f\big(X^{1}_{n}(t),\dots,X^{n}_{n}(t)\big)=\Gs_n f\big(\Ss_n(t),\Ts_n(t)\big),\]
where, for any $(x,y)\in \R^2$,
\begin{multline*}
\Gs_nf(x,y)=-\frac{\sqrt{n}y}{\s^2}\frac{\partial f}{\partial y}(x,y)-\frac{n^{1/4}xy}{2\s^4}\frac{\partial f}{\partial x}(x,y)\\
+\frac{1}{2\s^6}\left(xy^2-x^3\s^2\right)\frac{\partial f}{\partial x}(x,y)+\frac{1}{2}\frac{\partial^2 f}{\partial x^2}(x,y)+2\s^2\frac{\partial^2 f}{\partial y^2}(x,y)+R_n^f(x,y)
\end{multline*}
with
\begin{multline*}
R_n^f(x,y)=\frac{\partial f}{\partial x}(x,y)\, R^{(1)}_n(x,y)+\frac{\partial f}{\partial y}(x,y)\, R^{(2)}_n(x,y)\\
+\frac{2x}{n^{1/4}}\frac{\partial^2 f}{\partial x \partial y}(x,y)+\frac{2y}{n^{1/4}}\frac{\partial^2 f}{\partial y^2}(x,y),
\end{multline*}
where $\smash{(R^{(1)}_n)_{n\geq 1}}$ and $\smash{(R^{(2)}_n)_{n\geq 1}}$ are sequences of functions from $\R^2$ to $\R$ verifying
\[\forall k>0 \qquad \lim_{n\to+\infty}\, \sup_{\gfrac{(x,y) \in \R^2}{\|(x,y)\|\leq k}}\,\max\,\left(\,|R^{(1)}_n(x,y)|\,,\,|R^{(2)}_n(x,y)|\,\right)=0.\]
\label{GenUn} 
\end{prop}

\begin{proof} Let us define $\Psi_f$ on $\R^n$ by
\[\forall x\in \R^n\qquad \Psi_f(x)=f\left(\frac{S_n[x]}{n^{3/4}},\frac{T_n[x]}{n^{3/4}}-n^{1/4}\s^2\right).\]
Let $x\in \R^n$. For any $i\in\{1,\dots,n\}$, we have
\[\frac{\partial \Psi_f(x)}{\partial y_j}=\frac{1}{n^{3/4}}\frac{\partial f}{\partial x}(\,\cdots)+\frac{2x_j}{n^{3/4}}\frac{\partial f}{\partial y}(\,\cdots),\]
\[\frac{\partial^2 \Psi_f(x)}{\partial y_j^2}=\frac{1}{n^{3/2}}\frac{\partial^2 f}{\partial x^2}(\,\cdots)+\frac{4x_j}{n^{3/2}}\frac{\partial^2 f}{\partial x\partial y}(\,\cdots)+\frac{4x_j^2}{n^{3/2}}\frac{\partial^2 f}{\partial y^2}(\,\cdots)+\frac{2}{n^{3/4}}\frac{\partial f}{\partial y}(\,\cdots),\]
where we write
\[(\,\cdots) \qquad \mbox{instead of}\qquad \left(\frac{S_n[x]}{n^{3/4}},\frac{T_n[x]}{n^{3/4}}-n^{1/4}\s^2\right)\]
in order to simplify the notations. We have then
\begin{multline*}
L_n\Psi_f(x)=\frac{1}{2}\sum_{j=1}^n\left[\frac{\partial^2 \Psi_f(x)}{\partial x_j^2}+\left(\!\frac{S_n[x]}{T_n[x]+1}-x_j\left(\!\frac{S_n[x]}{T_n[x]+1}\!\right)^2\!-\!\frac{x_j}{\s^2}\!\right)\!\frac{\partial \Psi_f(x)}{\partial x_j}\right]\\
=\frac{1}{2\sqrt{n}}\frac{\partial^2 f}{\partial x^2}(\,\cdots)+\frac{2S_n[x]}{n^{3/2}}\frac{\partial^2 f}{\partial x\partial y}(\,\cdots)+\frac{2T_n[x]}{n^{3/2}}\frac{\partial^2 f}{\partial y^2}(\,\cdots)\\
\qquad+\frac{1}{2}\left(\frac{n^{1/4}S_n[x]}{1+T_n[x]}-\frac{S_n^3[x]}{n^{3/4}(1+T_n[x])^2}-\frac{S_n[x]}{n^{3/4}\s^2}\right)\frac{\partial f}{\partial x}(\,\cdots)\\
\qquad+\left(n^{1/4}+\frac{S_n^2[x]}{n^{3/4}(1+T_n[x])^2}-\frac{T_n[x]}{n^{3/4}\s^2}\right)\frac{\partial f}{\partial y}(\,\cdots).
\end{multline*}
We obtain that
\[\sqrt{n}L_n\Psi_f(x)=\Gs_n f\left(\frac{S_n[x]}{n^{3/4}},\frac{T_n[x]}{n^{3/4}}-n^{1/4}\s^2\right),\]
where $\Gs_nf$ is defined on $\R^2$ by
\begin{multline*}
\forall (x,y)\in \R^2\qquad \Gs_nf(x,y)=\frac{2x}{n^{1/4}}\frac{\partial^2 f}{\partial x\partial y}(x,y)+\left(\frac{2y}{n^{1/4}}+2\s^2\right)\frac{\partial^2 f}{\partial y^2}(x,y)\\
+\frac{1}{2}\frac{\partial^2 f}{\partial x^2}(x,y)+\left(-\frac{\sqrt{n}x}{2\s^2}(1-h_n(y))-\frac{x^3}{2\s^4}h_n(y)^2\right)\frac{\partial f}{\partial x}(x,y)\\
+\left(-\frac{\sqrt{n}y}{\s^2}+\frac{x^2}{n^{3/4}\s^4}h_n(y)^2\right)\frac{\partial f}{\partial y}(x,y),
\end{multline*}
with
\[h_n : y\in \,]-\s^2n^{1/4},+\infty[\,\longmapsto\left(1+\frac{y}{n^{1/4}\s^2}+\frac{1}{n\s^2}\right)^{-1}.\]
We introduce the functions $\eps_n^{(1)}$ and $\eps_n^{(2)}$ such that
\[\forall y>-\s^2n^{1/4}\qquad h_n(y)=1-\frac{y}{n^{1/4}\s^2}+\frac{y^2}{\sqrt{n}\s^4}+\frac{1}{\sqrt{n}}\eps_n^{(1)}(y)\]
and $\eps_n^{(2)}(y)=h_n(y)^2-1$. We obtain the formula of $\Gs_nf$ given in the proposition with
\[R_n^{(1)}:(x,y)\longmapsto \frac{x}{2\s^2}\eps_n^{(1)}(y)-\frac{x^3}{2\s^4}\eps_n^{(2)}(y)\quad \mbox{and}\quad R_n^{(2)}:(x,y)\longmapsto \frac{x^2h_n(y)^2}{n^{3/4}\s^2}.\]
It is easy to see that $\smash{(R^{(1)}_n)_{n\geq 1}}$ and $\smash{(R^{(2)}_n)_{n\geq 1}}$ are sequences of functions which converge to $0$ uniformly over any compact set in $\R^2$.
\end{proof}

\subsection{Sketch of proof of theorem~\ref{MainTheoGauss}}
\label{sketch}

Let us denote by $G_{\s}$ the infinitesimal generator of the Markov process which is solution of $(\mathcal{S}_{\s})$. It is defined by
\[\forall f\in \Ck{2}(\R)\qquad\forall x \in \R\qquad G_{\s}f(x)=\frac{1}{2}f''(x)-\frac{x^3}{2\s^4}f'(x)\]

Let $n\geq 1$ and $f\in \Ck{2}(\R)$. By abuse of notation, we also write $f$ for the function $(x,y)\in \R^2\longmapsto f(x)$. The essential ingredient for the proof of theorem~\ref{MainTheoGauss} is the introduction of a suitable martingale problem. By It\^o's formula (see~\cite{SV}), we prove that
\[f\big(\Ss_n(t) \big)=f\big(\Ss_n(0) \big)+\int_{0}^t\Gs_nf\big(\Ss_n(s) \big)\,ds+\Mc_{n,f}(t),\]
where $\Mc_{n,f}$ is a local martingale. By proposition~\ref{GenUn}, we have
\[\Gs_nf\big(\Ss_n \big)=\underbrace{\left(-\frac{n^{1/4}\Ss_n\Ts_n}{2\s^4}+\frac{\Ss_n\Ts_n^2}{2\s^6}\right)f'\big(\Ss_n \big)}_{\widetilde{A}_f\big(\Ss_n,\Ts_n\big)}\,+\,G_{\s}f\big(\Ss_n \big)+f'\big(\Ss_n \big)R_n^{(1)}\big(\Ss_n,\Ts_n\big)\]
where $\smash{(R^{(1)}_n)_{n\geq 1}}$ is a sequence of functions which converges to $0$ uniformly over any compact set in $\R^2$.\medskip

\textbf{Step 1:} We notice that the term $\smash{\widetilde{A}_f\big(\Ss_n,\Ts_n\big)}$ does not converge a priori. To solve this problem, we introduce a perturbation: we transform the function $f$ into a function $F_{n,f}$ which converges to $f$ as $n$ goes to $\infty$, and which satisfies
\[\Gs_nF_{n,f}\big(\Ss_n,\Ts_n \big)=G_{\s}f\big(\Ss_n \big)+\mbox{a remainder}.\]
Notice that the perturbation theory and methodology was first introduced in~\cite{PapaSV}.\medskip

\textbf{Step 2:} For any $k\geq 1$, we define the stopping time $\tau_n^k$ as the first exit time of a path of $\big(\Ss_n,\Ts_n\big)$ from the domain $[-k,k]^2$, and we prove that $\smash{\Mc_{n,f}^k=\Mc_{n,f}(\,\cdot\,\wedge \tau_n^k\wedge T)}$ is a martingale which is bounded over $\Ll^2$, for any $T>0$ and $k\geq 1$.\medskip

\textbf{Step 3:} We prove that $\P(\t_n^{k}\leq T)$, the probability that a path of $\big(\Ss_n,\Ts_n\big)$ exits $[-k,k]^2$ before the time $T$, goes to $0$ when $n$ and $k$ goes to $+\infty$. We also use the concept of collapsing processes (see appendix) in order to prove that the sequence of processes $\big(\Ts_n(t),\,t\geq 0\big)_{n\geq 1}$ converges to $0$ in the following sense:
\[\forall\eta>0\qquad \lim_{n\to+\infty}\,\P\left(\,\sup_{0\leq t\leq T}\,\big|\Ts_n(t)\big|>\eta\,\right)=0.\]

\textbf{Step 4:} We prove that the sequence $(\Ss_n(t),\,t\geq 0)_{n\geq 1}$ is tight in the Skorokhod space $\Dc([0,T],\R)$.\medskip

\textbf{Step 5:} We deduce from the previous steps that there exists a subsequence $\smash{\big(\Ss_{m_n}\big)_{n\geq 1}}$ which converges in distribution to some process $\Uc$ on $\Dc([0,T],\R)$. We prove then that, for any $k\geq 1$ and $t\in [0,T]$,
\[\Mc^k_{m_n,f}(t)\overset{\loi}{\underset{n\to +\infty}{\longrightarrow}}\Mc_f(t)=f(\Uc(t\wedge T))-f(\Uc(0))-\int_0^{t\wedge T} G_{\s}f(\Uc(s))\,ds,\]
and that $\Mc_f$ is a martingale. As a consequence $\Uc$ is uniquely determined as the unique solution of the martingale problem associated to $G_{\s}$. We conclude that $\Uc$ is the solution of $(\Sc_{\s})$ and that $\smash{\big(\Ss_n\big)_{n\geq 1}}$ converges in distribution to $\Uc$ on $\Dc([0,T],\R)$, and thus on $C([0,T],\R)$.\medskip

These steps are developed in detail in the next section.

\section{Proof of theorem~\ref{MainTheoGauss}}
\label{proof}

\subsubsection{Perturbation}

Let $f \in \Ck{2}(\R)$. We want to find functions $H_f$ and $K_f$ defined on $\R^2$ such that
\[F_{n,f}:(x,y)\longmapsto f(x)+\frac{1}{n^{1/4}}H_f(x,y)+\frac{1}{\sqrt{n}}K_f(x,y),\]
satisfies
\[\Gs_n F_{n,f}=G_{\s}f+\widetilde{R}_{n,f},\]
where $\widetilde{R}_{n,f}$ is a remainder term. Let us find necessary conditions. We suppose that we have built $H_f$ and $K_f$ and we assume that they are~$\Ck{2}$. We have then, for any $(x,y)\in \R^2$,
\begin{multline*}
\Gs_{n}F_{n,f}(x,y)=n^{1/4}\left(-\frac{y}{\s^2}\frac{\partial H_f}{\partial y}(x,y)-\frac{xy}{2\s^4}f'(x)\right)-\frac{y}{\s^2}\frac{\partial K_f}{\partial y}(x,y)\\
-\frac{xy}{2\s^4}\frac{\partial H_f}{\partial x}(x,y)+\frac{1}{2\s^6}\left(xy^2-x^3\s^2\right)f'(x)+\frac{1}{2}f''(x)+\mbox{a remainder}\,.
\end{multline*}
The function $H_f$ should verify
\[\forall (x,y)\in \R^2\qquad -\frac{y}{\s^2}\frac{\partial H_f}{\partial y}(x,y)-\frac{xy}{2\s^4}f'(x)=0.\]
We choose
\[H_f:(x,y)\longmapsto-\frac{xy}{2\s^2}f'(x).\]
Therefore the function $K_f$ should satisfy, for all $(x,y)\in \R^2$,
\begin{align*}
\Gs_{n}F_{n,f}(x,y)&=-\frac{y}{\s^2}\frac{\partial K_f}{\partial y}(x,y)+\frac{xy^2}{4\s^6}(f'(x)+xf''(x))\\
&\quad\qquad\qquad+\frac{1}{2\s^6}\left(xy^2-x^3\s^2\right)f'(x)+\frac{1}{2}f''(x)+\mbox{the remainder}\\
&=-\frac{y}{\s^2}\frac{\partial K_f}{\partial y}(x,y)+\frac{xy^2}{4\s^6}(3f'(x)+xf''(x))\\
&\qquad\qquad\qquad\qquad-\frac{x^3}{2\s^4}f'(x)+\frac{1}{2}f''(x)+\mbox{the remainder}\,.
\end{align*}
So that the variable $y$ disappears in the leading term of $\Gs_{n}F_{n,f}(x,y)$, the function $K_f$ should verify
\[\forall (x,y)\in \R^2\qquad-\frac{y}{\s^2}\frac{\partial K_f}{\partial y}(x,y)+\frac{xy^2}{4\s^6}(3f'(x)+xf''(x))=0.\]
We choose
\[K_f:(x,y)\longmapsto \frac{xy^2}{8\s^4}(3f'(x)+xf''(x)).\]
It is easy to see that these choices for $H_f$ and $K_f$ are sufficient for the variable~$y$ to disappear in the leading term of $\Gs_{n}F_{n,f}(x,y)$. The remainder term is then
\[\widetilde{R}_{n,f}=R_n^f+\frac{1}{n^{1/4}}R_n^{H_f}+\frac{1}{\sqrt{n}}R_n^{K_f}.\]
We notice that, so that the above computations are possible, it is necessary that $f$ is $\Ck{4}$. Indeed, the first four derivatives of $f$ appear in the remainder term. We also remark that, if $f\in \Ck{4}(\R)$, then the functions 
$H_f$, $K_f$ and their first and second derivatives are bounded over any compact set in $\R^2$. Finally let us recall that $\smash{(R^{(1)}_n)_{n\geq 1}}$ and $\smash{(R^{(2)}_n)_{n\geq 1}}$ are sequences of functions which converge to $0$ when $n$ goes to $+\infty$, uniformly over any compact set. As a consequence we have the following proposition:

\begin{prop} Let $n\geq 1$ and $f \in \Ck{4}(\R)$. We define $H_f$ and $K_f$ on $\R^2$ by
\[\forall (x,y)\in \R^2\qquad H_f(x,y)=-\frac{xy}{2\s^2}f'(x),\qquad K_f(x,y)=\frac{xy^2}{8\s^4}(3f'(x)+xf''(x)).\]
Then the function
\[F_{n,f}:(x,y)\longmapsto f(x)+\frac{1}{n^{1/4}}H_f(x,y)+\frac{1}{\sqrt{n}}K_f(x,y),\]
verifies $\Gs_n F_{n,f}=G_{\s}f+\widetilde{R}_{n,f}$, with $\widetilde{R}_{n,f}$ a remainder term satisfying
\[\forall k>0 \qquad \lim_{n\to+\infty}\, \sup_{\gfrac{(x,y) \in \R^2}{\|(x,y)\|\leq k}}\,\left|\widetilde{R}_{n,f}(x,y)\right|=0.\]
\label{GenS_n()}
\end{prop}

\subsubsection{Introduction of a martingale problem}

We give ourselves $n\geq 1$ and $f \in \Ck{4}(\R)$. For any $t\geq 0$, we have
\[f\left(\frac{S_n(\sqrt{n}t)}{n^{3/4}}\right)=f\big(\Ss_n(t)\big)=\left(F_{n,f}-\frac{1}{n^{1/4}}H_f-\frac{1}{\sqrt{n}}K_f\right)\big(\Ss_n(t),\Ts_n(t)\big).\]

We define the process $\smash{(\Mc_{n,f}(t),\,t\geq 0)}$ by
\begin{multline*}
\forall t\geq 0 \qquad \Mc_{n,f}(t)=F_{n,f}\big(\Ss_n(t),\Ts_n(t)\big)-F_{n,f}\big(\Ss_n(0),\Ts_n(0)\big)\\-\int_0^t\Gs_n F_{n,f}\big(\Ss_n(s),\Ts_n(s)\big)\,ds.
\end{multline*}
By applying It\^o's formula to the function
\[\Psi_{n,f}: (x_1,\dots,x_n)\in \R^n\longmapsto F_{n,f}\left(\frac{S_n[x]}{n^{3/4}},\frac{T_n[x]}{n^{3/4}}-n^{1/4}\s^2\right),\]
we obtain
\[\forall t\geq 0 \qquad \Mc_{n,f}(t)=n^{1/4}\sum_{j=1}^n\int_0^t \frac{\partial \Psi_{n,f}}{\partial x_j}\big(X_n(\sqrt{n} s)\big)\,dB_j(s).\]
It is a local martingale and
\[\forall t\geq 0 \qquad \langle \Mc_{n,f},\Mc_{n,f}\rangle_t=\sqrt{n}\sum_{j=1}^n\int_0^{t}\left(\frac{\partial \Psi_{n,f}}{\partial x_j}\right)^2\big(X_n(\sqrt{n}s)\big)\,ds.\]
For any $k>0$, we introduce the stopping time $\t_n^k$ defined by
\[\t_n^k=\inf_{t\geq 0}\,\Big\{\,\,\big|\Ss_n(t)\big|\geq k\quad\mbox{or}\quad   \big|\Ts_n(t)\big|\geq k\,\Big\}.\]
Let $T>0$. We denote $\smash{\Mc_{n,f}^k(t)=\Mc_{n,f}(t\wedge \t_n^k\wedge T)}$ for any $t\geq 0$.

\begin{lem} For all $k\geq 1$, $n\geq 1$ and $f\in \Ck{4}(\R)$, the process $\smash{\Mc_{n,f}^k}$ is a martingale which is bounded over~$\Ll^2$. Moreover
\[\forall t\geq 0\qquad\sup_{n\geq 1}\,\E\left(\Mc_{n,f}^k(t)^2\right)<+\infty.\]
\label{Mnf-martingale}
\end{lem}

\begin{proof} For any $t\geq 0$, we have
\[\langle \Mc_{n,f}^k,\Mc_{n,f}^k\rangle_t=\sqrt{n}\sum_{j=1}^n\int_0^{t\wedge\t_n^k\wedge T}\left(\frac{\partial \Psi_{n,f}}{\partial x_j}\right)^2\big(X_n(\sqrt{n}s)\big)\,ds.\]
Moreover, for all $i\in \{1,\dots,n\}$ and $x\in \R^n$,
\begin{equation}
\begin{split}
\frac{\partial \Psi_{n,f}}{\partial x_i}&(x)=\frac{1}{n^{3/4}}f'\left(\frac{S_n[x]}{n^{3/4}}\right)\\
&\begin{split}
+\frac{1}{n^{3/4}}&\left(\frac{1}{n^{1/4}}\frac{\partial H_f}{\partial x}+\frac{1}{n^{1/2}}\frac{\partial K_f}{\partial x}\right)\left(\frac{S_n[x]}{n^{3/4}},\frac{T_n[x]}{n^{3/4}}-n^{1/4}\s^2\right)\\
&+\frac{2x_i}{n^{3/4}}\left(\frac{1}{n^{1/4}}\frac{\partial H_f}{\partial y}+\frac{1}{n^{1/2}}\frac{\partial K_f}{\partial y}\right)\left(\frac{S_n[x]}{n^{3/4}},\frac{T_n[x]}{n^{3/4}}-n^{1/4}\s^2\right).
\end{split}
\end{split}
\label{Psi(n,f)}
\end{equation}
%\begin{multline*}
%\frac{\partial \Psi_{n,f}}{\partial x_i}(x)=\frac{1}{n^{3/4}}f'\left(\frac{S_n[x]}{n^{3/4}}\right)\\
%+\frac{1}{n^{3/4}}\left(\frac{1}{n^{1/4}}\frac{\partial H_f}{\partial x}+\frac{1}{n^{1/2}}\frac{\partial K_f}{\partial x}\right)\left(\frac{S_n[x]}{n^{3/4}},\frac{T_n[x]}{n^{3/4}}-n^{1/4}\s^2\right)\\
%+\frac{2x_i}{n^{3/4}}\left(\frac{1}{n^{1/4}}\frac{\partial H_f}{\partial y}+\frac{1}{n^{1/2}}\frac{\partial K_f}{\partial y}\right)\left(\frac{S_n[x]}{n^{3/4}},\frac{T_n[x]}{n^{3/4}}-n^{1/4}\s^2\right).
%\end{multline*}
By squaring these terms and by summing over all $i\in \{1,\dots,n\}$, we observe that there exists a constant $C_f^k>0$ such that, for all $x\in \R^n$ verifying
\[\qquad \left|\frac{S_n[x]}{n^{3/4}}\right|<k\qquad\mbox{and}\qquad \left|\frac{T_n[x]}{n^{3/4}}-n^{1/4}\s^2\right|<k,\]
we have
\[\sum_{j=1}^n\left(\frac{\partial \Psi_{n,f}}{\partial x_j}\right)^2\leq \frac{C_f^k}{\sqrt{n}}.\]
As a consequence, for any $t\geq 0$,
\[\sup_{n\geq 1}\,\E\left(\langle \Mc_{n,f}^k,\Mc_{n,f}^k\rangle_t\right)\leq C_f^k\,T.\]
Therefore, for any $n\geq 1$, the process $\smash{\Mc_{n,f}^k}$ is a martingale bounded over~$\Ll^2$ (see theorem~4.8~of~\cite{LeGall}) and
\[\forall t\geq 0\qquad\E\left(\Mc_{n,f}^k(t)^2\right)=\E\left(\langle \Mc_{n,f}^k,\Mc_{n,f}^k\rangle_t\right)\leq C_f^k\,T.\]
This ends the proof of the lemma.
\end{proof}

\subsubsection{Study of the asymptotic behavior $(\t_n^k)_{n\geq 1}$}

\begin{lem} For any $\eps>0$, there exist $n_{\eps}\geq 1$ and $k_{\eps}\geq 1$ such that
\[\sup_{n\geq n_{\eps}}\,\P\left(\t_n^{k_{\eps}}\leq T\right)\leq \eps.\]
Moreover the process $\big(\Ts_n(t),\,t\geq 0\big)_{n\geq 1}$ collapses to zero, i.e.,
\[\forall\eta>0\qquad \lim_{n\to+\infty}\,\P\left(\,\sup_{0\leq t\leq T}\,\big|\Ts_n(t)\big|>\eta\,\right)=0.\]
\label{LemArret}
\end{lem}

\begin{proof} Let $k,\eps>0$ and $n\geq 1$. We have
\[\P\left(\t_n^{k}\leq T\right)\leq \P\left(\,\sup_{0\leq t\leq T\wedge \t_n^k}\,\big|\Ts_n(t)\big|\geq \frac{k}{2}\,\right)+\P\left(\,\sup_{0\leq t\leq T\wedge \t_n^k}\,\big|\Ss_n(t)\big|\geq \frac{k}{2}\,\right).\]
We denote $\P(A_n^k)+\P(B_n^k)$ the sum in the right side of this inequality.\medskip

Let us deal with the bound of $\P(A_n^k)$. To this end we would like to apply proposition~\ref{CollapsingProp} in appendix to the positive semimartingale $(\xi_n(t),\,t\geq 0)_{n \geq 1}$ defined by
\[\forall n\geq 1\quad\forall t\geq 0\qquad \xi_n(t)=\Ts_n(t)^2.\]
By applying It\^o's formula, we get
\[d\xi_n(t)=\Gs_n f_0\big(\Ss_n(t),\Ts_n(t)\big)\,dt+n^{1/4}\sum_{i=1}^n \frac{4X_n^i(\sqrt{n}t)}{n^{3/4}}\,\Ts_n(t)\,dB_i(t),\]
with $f_0:(x,y)\longmapsto y^2$. With the notations of proposition~\ref{CollapsingProp}, we have $\smash{\z_n(t)=\Gs_n f_0\big(\Ss_n(t),\Ts_n(t)\big)}$ and $Z_{n,i}(t)=4n^{-1/2} X_n^i(\sqrt{n}t)\,\Ts_n(t)$ for all $t\geq 0$, $n\geq 1$ and $i\in \{1,\dots,n\}$. We have
\begin{align*}
\forall n\geq 1\quad\forall t\in [0,\tau_n^k]\qquad \sum_{i=1}^{n} Z_{n,i}(t)^2 &=16\,\Ts_n(t)^2\,\frac{1}{n}\sum_{i=1}^n X_n^i(\sqrt{n}t)^2\\
&=16\,\Ts_n(t)^2\,\left(\s^2+\frac{\Ts_n(t)}{n^{1/4}}\right).
\end{align*}
Hence condition~(\ref{C4}) of proposition~\ref{CollapsingProp} is verified with $C_5=16k^2(\s^2+k)$. Next, by proposition~\ref{GenUn}, for any $n\geq 1$ and $t\in [0,\tau_n^k]$
\begin{align*}
\z_n(t)&=-\frac{2\sqrt{n}}{\s^2}\,\Ts_n(t)^2+4\s^2+2\,\Ts_n(t){R}^{(2)}_n\left(\Ss_n(t),\Ts_n(t)\right)+\frac{4}{n^{1/4}}\Ts_n(t)\\
&\leq -\frac{2\sqrt{n}}{\s^2}\,\xi_n(t)+4\s^2+2k\,\sup_{\|(x,y)\|\leq k}\,\left|{R}^{(2)}_n(x,y)\right|+\frac{4k}{n^{1/4}}.
\end{align*}
Condition~(\ref{C3}) is then verified with $\kappa_n=\sqrt{n}$ for any $n\geq 1$, $C_2=2/\s^2$,
\[C_4=4\s^2+2k\,\sup_{n\geq 1}\,\sup_{\|(x,y)\|\leq k}\,\left|{R}^{(2)}_n(x,y)\right|+4k<+\infty\]
and $C_3$, $(\beta_n)_{n\geq1}$ may be chosen arbitrarily. We choose $(\beta_n)_{n\geq1}$ such that $\b_n/\k_n$ goes to $0$ when $n$ goes to $+\infty$.\medskip

Let us examine condition (\ref{C2}): we denote $\smash{Y_n^i=\big(X_n^i(0)\big)^2-\s^2}$ for any\linebreak $i\in \{1,\dots,n\}$. Since $\smash{X_n^1(0),\dots,X_n^n(0)}$ are independent random variables with common distribution $\Nc(0,\s^2)$, we get that $Y_n^1,\dots,Y_n^n$ are independent identically distributed random variables which are centered and have finite moments of all orders. Theorem~2~of~\cite{Brillinger} implies that, for any $v\geq 2$, there exists $K_v>0$ such that
\[\forall n\geq 1\qquad \E \left( \Big|Y_n^1+\dots+Y_n^n\big|^v\right)\leq K_v n^{v/2}. \]
Hence, for all $d>1$ and $n\geq 1$, 
\[\E \left[ \Big( \xi_n(0) \Big)^d \right]=\E\left[ \left( \frac{1}{n^{3/4}}\Big(Y_n^1+\dots+Y_n^n\big) \right)^{2d} \right] \leq K_{2d} \,\frac{n^{d}}{n^{3d/2}}= K_{2d} \, n^{-d/2}.\]
Condition~(\ref{C2}) is then satisfied for any $d>1$, with $C_1=K_{2d}$ and $\a_n\leq \sqrt{n}$ for all $n\geq 1$. So that condition~(\ref{C1}) is verified, we choose $d>2$ and $\a_n=n^{1/4}$ for all $n\geq 1$. We have
\[\kappa_n^{\frac{1}{d}} \alpha_n^{-1} \vee \alpha_n \kappa_n^{-1}=n^{1/(2d)-1/4}\vee n^{-1/4}=n^{1/(2d)-1/4}.\]
As a consequence, proposition~\ref{CollapsingProp} implies that there exist $M>0$ and $n_1\geq 1$ such that
\begin{equation}
\sup_{n \geq n_1} \,\P \left(\, \sup_{0 \leq t \leq T \wedge \tau_n^k} \big|\Ts_n(t)\big|^2 > M n^{1/(2d)-1/4}\,\right) \leq \frac{\eps}{2}.
\label{Eq(ConsCollapProc)}
\end{equation}
We increase the value of $n_1$ so that
\[\sup_{n \geq n_1} \,\P \left(\, \sup_{0 \leq t \leq T \wedge \tau_n^k} \big|\Ts_n(t)\big|^2 > \frac{k}{2}\,\right) \leq \frac{\eps}{2}.\]

Let us deal now with the term $\P(B_n^k)$. In the rest of this proof, we assume that $f$ is the function $(x,y)\longmapsto x^2$. We have
\[\forall n\geq 1\qquad\Ss_n(t)^2=F_{n,f}\big(\Ss_n(t),\Ts_n(t)\big)+\frac{\Ss_n(t)^2 \,\Ts_n(t)}{n^{1/4}\s^2}-\frac{\Ss_n(t)^2 \,\Ts_n(t)^2}{\sqrt{n}\s^4},\]
thus
\[\forall n\geq 1\qquad F_{n,f}\big(\Ss_n(t),\Ts_n(t)\big)=\Ss_n(t)^2\left(1-\frac{\Ts_n(t)}{n^{1/4}\s^2}+\frac{\Ts_n(t)^2}{\sqrt{n}\s^4}\right).\]
We obtain that, for $n$ large enough,
\begin{multline*}
\P(B_n^k)=\P \left(\, \sup_{0 \leq t \leq T \wedge \tau_n^k} \big|\Ss_n(t)\big|^2 > \frac{k^2}{4}\,\right)\\
\leq \P \left(\, \sup_{0 \leq t \leq T \wedge \tau_n^k} F_{n,f}\big(\Ss_n(t),\Ts_n(t)\big) > \frac{k^2}{8}\,\right)\\
\leq \P \left(F_{n,f}\big(\Ss_n(0),\Ts_n(0)\big) > \frac{k^2}{24}\,\right)+\P \left(\, \sup_{0 \leq t \leq T \wedge \tau_n^k} \Mc_{n,f}(t) > \frac{k^2}{24}\,\right)\\
+\P \left(\, \sup_{0 \leq t \leq T \wedge \tau_n^k} \Gs_n F_{n,f}\big(\Ss_n(t),\Ts_n(t)\big) > \frac{k^2}{24T}\,\right).
\end{multline*}
For any $n\geq 1$, the random variables $X_n^1(0),\dots,X_n^n(0)$ are independent with common distribution $\Nc(0,\s^2)$ thus, by the Central Limit Theorem, we get $(\Ss_n(0))_{n\geq 1}$ and $(\Ts_n(0))_{n\geq 1}$ converge in distribution to $0$. This implies that, for~$n$ large enough,
\[\P \left(F_{n,f}\big(\Ss_n(0),\Ts_n(0)\big) > \frac{k^2}{24}\,\right)\leq \frac{\eps}{6}.\]
Next proposition~\ref{GenS_n()} gives us
\begin{align*}
\Gs_n F_{n,f}\big(\Ss_n(t),\Ts_n(t)\big)&=1-\frac{\Ss_n(t)^4}{\s^4}+\widetilde{R}_{n,f}\big(\Ss_n(t),\Ts_n(t)\big)\\
&\leq 1+\left|\widetilde{R}_{n,f}\big(\Ss_n(t),\Ts_n(t)\big)\right|.
\end{align*}
and
\[\lim_{n\to+\infty}\, \sup_{\|(u,v)\|\leq k}\,\left|\smash{\widetilde{R}_{n,f}}(u,v)\right|=0.\]
If we choose $k>\sqrt{24T}$ and $n$ large enough, then
\begin{multline*}
\P \left(\, \sup_{0 \leq t \leq T \wedge \tau_n^k} \Gs_n F_{n,f}\big(\Ss_n(t),\Ts_n(t)\big) > \frac{k^2}{24T}\,\right)\\
\leq \P \left(1+\sup_{\|(u,v)\|\leq k}\,\left|\smash{\widetilde{R}_{n,f}}(u,v)\right| > \frac{k^2}{24T}\,\right)\leq \frac{\eps}{6}.
\end{multline*}
Finally, by lemma~\ref{Mnf-martingale}, $\smash{\Mc_{n,f}^k}$ is a martingale thus Doob's maximal inequality implies
\[\P \left(\, \sup_{0 \leq t \leq T \wedge \tau_n^k} \Mc_{n,f}(t) > \frac{k^2}{24}\,\right)\leq \frac{\E\left(\Mc_{n,f}^k(T)^2\right)}{(k^2/24)^2}.\]
Lemma~\ref{Mnf-martingale} also implies that $\smash{\big(\E\big(\Mc_{n,f}^k(T)^2\big)\big)_{n\geq 1}}$ is a bounded sequence. Hence, for $k$ large enough,
\[\P \left(\, \sup_{0 \leq t \leq T \wedge \tau_n^k} \Mc_{n,f}(t) > \frac{k^2}{24}\,\right)\leq \frac{\eps}{6}.\]
As a consequence, there exist $n_2\geq 1$ and $k_{\eps}\geq 1$ such that $\P(B_n^{k_{\eps}})\leq \eps/2$ for all $n\geq n_2$. We denote $n_{\eps}=n_1\vee n_2$. We have proved that
\[\forall n\geq n_{\eps}\qquad \P\left(\t_n^{k_{\eps}}\leq T\right)\leq \P(A_n^{k_{\eps}})+\P(B_n^{k_{\eps}})\leq \eps.\]

Let us prove the second assertion of the lemma: for any $\eta>0$, we have
\[\P\left(\,\sup_{0\leq t\leq T}\,\big|\Ts_n(t)\big|>\eta\,\right)\leq\P\left(\,\sup_{0\leq t\leq T \wedge \tau_n^{k_{\eps}}}\,\big|\Ts_n(t)\big|^2>\eta^2\,\right)+\P\left(\t_n^{k_{\eps}}\leq T\right).\]
By formula~(\ref{Eq(ConsCollapProc)}), for $n$ large enough,
\[\P\left(\,\sup_{0\leq t\leq T}\,\big|\Ts_n(t)\big|>\eta\,\right)\leq\frac{\eps}{2}+\P\left(\t_n^{k_{\eps}}\leq T\right)\leq \frac{3\eps}{2}.\]
By letting $\eps$ goes to $0$, we obtain that $\smash{\big(\Ts_n(t),\,t\geq 0\big)_{n\geq 1}}$ collapses to zero. This ends the proof of the lemma.
\end{proof}

\subsubsection{Tightness of $(\Ss_n(t),\,t\geq 0)_{n\geq 1}$ in $\Dc([0,T],\R)$}

Since $(X_n(t),\,0\leq t\leq T)$, $n\geq 1$, and the limiting process $(\Uc(t),\,0\leq t\leq T)$ belong to $C([0,T],\R)$, it is enough to prove that
$\smash{(\Ss_n(t),\,t\geq 0)_{n\geq 1}}$ is relatively compact for the weak convergence in $\Dc([0,T],\R)$, which is a Polish space (see~theorem~12.2 of~\cite{BillCVPM}). Prohorov theorem (theorem~5.1 of~\cite{BillCVPM}) implies that it is enough to prove that $\smash{(\Ss_n(t),\,t\geq 0)_{n\geq 1}}$ is a tight sequence. As in~\cite{ComEis} and~\cite{DaiPra}, we use the following tightness criterion:\medskip

\begin{prop} A sequence $(\xi_n(t),\,0\leq t\leq T)_{n\geq 1}$ on $\Dc([0,T],\R)$ is tight if\smallskip

\textbf{(a)} for any $\eps>0$, there exists $M>0$ such that
\[\sup_{n\geq 1}\,\P\left(\,\sup_{0\leq t\leq T}\,\big|\xi_n(t)\big|\geq M\,\right)\leq \eps,\]

\textbf{(b)} for any $\eps>0$ and $\eta>0$, there exists $\d>0$ such that
\[\sup_{n\geq 1}\,\sup_{\gfrac{\t_1,\t_2 \in \Tc_n}{0\leq \t_1\leq\t_2\leq (\t_1+\d)\wedge T}}\,\P\left(\,\big|\xi_n(\t_2)-\xi_n(\t_1)\big|\geq \eta\,\right)\leq \eps,\]
where, for any $n\geq 1$, $\Tc_n$ is the set of all the stopping times adapted to the filtration generated by the process $\xi_n$.
\label{CritereTension}
\end{prop}\medskip

\begin{lem} The sequence $\left(\Ss_n(t),\,0\leq t\leq T\right)_{n\geq 1}$ is relatively compact for the weak convergence on $\Dc([0,T],\R)$.
\label{LemTension}
\end{lem}

\begin{proof} It is enough to prove that $\left(\Ss_n(t),\,0\leq t\leq T\right)_{n\geq 1}$ verifies conditions $(a)$ and $(b)$ of proposition~\ref{CritereTension}. In the proof of lemma~\ref{LemArret}, we proved that, for any $\a>0$, there exists $k_{\a}>0$ and $n_{\a}\geq 1$ such that
\[\sup_{n\geq n_{\a}}\,\P\left(\t_n^{k_{\a}}\leq T\right)\leq \a\]
and, for all $n\geq n_{\a}$,
\[\P \left(\, \sup_{0 \leq t \leq T \wedge \tau_n^{k_{\a}}} \big|\Ss_n(t)\big| > \frac{k_{\a}}{2}\,\right)\leq\frac{\a}{2}. \]
We give ourselves $\eps>0$ and we denote $\a=2\eps/3$. We obtain that, for all $n\geq n_{\a}$,
\begin{multline*}
\P \left(\, \sup_{0 \leq t \leq T}\, \big|\Ss_n(t)\big| > \frac{k_{\a}}{2}\,\right)\\
\leq \P \left(\, \sup_{0 \leq t \leq T \wedge \tau_n^{k_{\eps}}}\big|\Ss_n(t)\big| > \frac{k_{\a}}{2}\,\right)+\P\left(\t_n^{k_{\a}}\leq T\right)\leq\frac{3\a}{2}=\eps.
\end{multline*}
Hence condition~$(a)$ is verified.\medskip

We prove now condition~$(b)$: we give ourselves $n\geq 1$ and $\eps,\eta,\d>0$. Let $\t_1$ and~$\t_2$ be two stopping times adapted to the filtration generated by the process~$\Ss_n$ and such that $0\leq \t_1\leq\t_2\leq (\t_1+\d)\wedge T$. Setting $\a=2\eps/3$, we have
\begin{align*}
\P\big(\big|\Ss_n(\t_2)&-\Ss_n(\t_1)\big|\geq \eta\big)\\
&\leq \P\left(\big|\Ss_n(\t_2\wedge \t_n^{k_{\a}})-\Ss_n(\t_1\wedge \t_n^{k_{\a}})\big|\geq \eta\right)+\P\left(\t_n^{k_{\a}}\leq T\right)\\
& \leq\frac{1}{\eta}\,\E\left(\,\big|\Ss_n(\t_2\wedge \t_n^{k_{\a}})-\Ss_n(\t_1\wedge \t_n^{k_{\a}})\big|\,\right)+\a,
\end{align*}
where we used Markov's inequality. In the rest of this proof, we assume that $f$ is the function $(x,y)\longmapsto x$. We have
\begin{multline*}
\big|\Ss_n(\t_2\wedge \t_n^{k_{\a}})-\Ss_n(\t_1\wedge \t_n^{k_{\a}})\big|\\
\leq\frac{1}{n^{1/4}}\left|H_f\big(\Ss_n(\t_2\wedge \t_n^{k_{\a}}),\Ts_n(\t_2\wedge \t_n^{k_{\a}})\big)-H_f\big(\Ss_n(\t_1\wedge \t_n^{k_{\a}}),\Ts_n(\t_1\wedge \t_n^{k_{\a}})\big)\right|\\
+\frac{1}{\sqrt{n}}\left|K_f\big(\Ss_n(\t_2\wedge \t_n^{k_{\a}}),\Ts_n(\t_2\wedge \t_n^{k_{\a}})\big)-K_f\big(\Ss_n(\t_1\wedge \t_n^{k_{\a}}),\Ts_n(\t_1\wedge \t_n^{k_{\a}})\big)\right|\\
+\int_{\t_1\wedge \t_n^{k_{\a}}}^{\t_2\wedge \t_n^{k_{\a}}}\left|\Gs_n F_{n,f}\big(\Ss_n(u),\Ts_n(u)\big)\right|\,du+\left|\Mc_{n,f}^{k_{\a}}(\t_2)-\Mc_{n,f}^{k_{\a}}(\t_1)\right|.
\end{multline*}
We have
\begin{align*}
\E\left(\,\left|\Mc_{n,f}^{k_{\a}}(\t_2)-\Mc_{n,f}^{k_{\a}}(\t_1)\right|\,\right)^2
&\leq \E\left(\,\left(\Mc_{n,f}^{k_{\a}}(\t_2)-\Mc_{n,f}^{k_{\a}}(\t_1)\right)^2\,\right)\\
&=\E\left(\,\E\bigg[\,\left(\Mc_{n,f}^{k_{\a}}(\t_2)-\Mc_{n,f}^{k_{\a}}(\t_1)\right)^2\,\bigg|\,\Gc_{\t_1}^n\,\bigg]\,\right),
\end{align*}
where $\smash{\Gc_t^{n}=\s\big(\Mc_{n,f}^{k_{\a}}(s),\,0\leq s\leq t\big)}$ for all $t\geq 0$. By lemma~\ref{Mnf-martingale}, $\smash{\Mc_{n,f}^{k_{\a}}}$ is a martingale bounded over $\Ll^2$ thus it is uniformly integrable. Martingale Stopping Theorem (theorem~3.16~of~\cite{LeGall}) implies that
\[\Mc_{n,f}^{k_{\a}}(\t_1)=\E\bigg[\,\Mc_{n,f}^{k_{\a}}(\t_2)\,\bigg|\,\Gc_{\t_1}^n\,\bigg].\]
Hence
\begin{multline*}
\E\bigg[\,\left(\Mc_{n,f}^{k_{\a}}(\t_2)-\Mc_{n,f}^{k_{\a}}(\t_1)\right)^2\,\bigg|\,\Gc_{\t_1}^n\,\bigg]\\=\E\bigg[\,\Mc_{n,f}^{k_{\a}}(\t_2)^2\,\bigg|\,\Gc_{\t_1}^n\,\bigg]+\Mc_{n,f}^{k_{\a}}(\t_1)^2-2\Mc_{n,f}^{k_{\a}}(\t_1)\E\bigg[\,\Mc_{n,f}^{k_{\a}}(\t_2)\,\bigg|\,\Gc_{\t_1}^n\,\bigg]\\
=\E\bigg[\,\Mc_{n,f}^{k_{\a}}(\t_2)^2\,\bigg|\,\Gc_{\t_1}^n\,\bigg]-\Mc_{n,f}^{k_{\a}}(\t_1)^2
\end{multline*}
and thus
%\begin{multline*}
%\E\bigg[\left(\Mc_{n,f}^{k_{\a}}(\t_2)-\Mc_{n,f}^{k_{\a}}(\t_1)\right)^2\bigg|\Gc_{\t_1}^n\bigg]=\E\bigg[\langle \Mc_{n,f}^{k_{\a}}, \Mc_{n,f}^{k_{\a}}\rangle_{\t_2}-\langle \Mc_{n,f}^{k_{\a}}, \Mc_{n,f}^{k_{\a}}\rangle_{\t_1}\bigg|\Gc_{\t_1}^n\bigg]\\
%=\E\bigg(\sqrt{n}\sum_{j=1}^n\int_{\t_1\wedge\t_n^{k_{\a}}}^{\t_2\wedge\t_n^{k_{\a}}}\left(\frac{\partial F_{n,f}}{\partial x_j}\right)^2\big(X_n(\sqrt{n}u)\big)\,du\,\bigg|\, \Fc_s\bigg)\leq C_f^{k_{\a}} \d
%\end{multline*}
\begin{multline*}
\E\bigg(\E\bigg[\,\left(\Mc_{n,f}^{k_{\a}}(\t_2)-\Mc_{n,f}^{k_{\a}}(\t_1)\right)^2\,\bigg|\,\Gc_{\t_1}^n\,\bigg]\bigg)\\
=\E\bigg(\E\bigg[\,\langle \Mc_{n,f}^{k_{\a}}, \Mc_{n,f}^{k_{\a}}\rangle_{\t_2}-\langle \Mc_{n,f}^{k_{\a}}, \Mc_{n,f}^{k_{\a}}\rangle_{\t_1}\,\bigg|\,\Gc_{\t_1}^n\,\bigg]\bigg)\\
=\E\bigg(\sqrt{n}\sum_{j=1}^n\int_{\t_1\wedge\t_n^{k_{\a}}}^{\t_2\wedge\t_n^{k_{\a}}}\left(\frac{\partial F_{n,f}}{\partial x_j}\right)^2\big(X_n(\sqrt{n}u)\big)\,du\bigg)\leq C_f^{k_{\a}} \d,
\end{multline*}
where $C_f^{k_{\a}}$ is the constant introduced in the proof of lemma~\ref{Mnf-martingale} for $k=k_{\a}$. We~get
\[\E\left(\,\left|\Mc_{n,f}^{k_{\a}}(\t_2)-\Mc_{n,f}^{k_{\a}}(\t_1)\right|\,\right)\leq \sqrt{C_f^{k_{\a}} \d}.\]
Next, since $f:(x,y)\longmapsto x$, proposition~\ref{GenS_n()} yields
\[\Gs_n F_{n,f}\big(\Ss_n(t),\Ts_n(t)\big)=-\frac{\Ss_n(t)^3}{2\s^4}+\smash{\widetilde{R}_{n,f}}\big(\Ss_n(t),\Ts_n(t)\big)\]
and
\[\forall k>0\qquad \lim_{n\to+\infty}\, \sup_{\|(x,y)\|\leq k}\,\big|\smash{\widetilde{R}_{n,f}}(x,y)\big|=0.\]
Therefore
\[\int_{\t_1\wedge\t_n^{k_{\a}}}^{\t_2\wedge\t_n^{k_{\a}}}\left|\Gs_n F_{n,f}\big(\Ss_n(u),\Ts_n(u)\big)\right|\,du\leq \left(\frac{k_{\a}^3}{2\s^4}+\sup_{\|(x,y)\|\leq k_{\a}}\,\big|\smash{\widetilde{R}_{n,f}}(x,y)\big|\right)\d.\]
Finally
\[\left|H_f\big(\Ss_n(\t_2\wedge \t_n^{k_{\a}}),\Ts_n(\t_2\wedge \t_n^{k_{\a}})\big)-H_f\big(\Ss_n(\t_1\wedge \t_n^{k_{\a}}),\Ts_n(\t_1\wedge \t_n^{k_{\a}})\big)\right|\leq \frac{k_{\a}^2}{\s^2}\]
and
\[\left|K_f\big(\Ss_n(\t_2\wedge \t_n^{k_{\a}}),\Ts_n(\t_2\wedge \t_n^{k_{\a}})\big)-K_f\big(\Ss_n(\t_1\wedge \t_n^{k_{\a}}),\Ts_n(\t_1\wedge \t_n^{k_{\a}})\big)\right|\leq \frac{3k_{\a}^3}{4\s^4}.\]
Hence, for $n$ large enough and $\d$ small enough,
\[\E\left(\,\big|\Ss_n(\t_2\wedge \t_n^{k_{\a}})-\Ss_n(\t_1\wedge \t_n^{k_{\a}})\big|\,\right)\leq \frac{\eta\a}{2}.\]
We obtain
\[\P\left(\,\big|\Ss_n(\t_2)-\Ss_n(\t_1)\big|\geq \eta\,\right)\leq \frac{3\a}{2}=\eps.\]
Condition~$(b)$ of proposition~\ref{CritereTension} is then satisfied and this ends the proof of the lemma.
\end{proof}

\subsubsection{Identification of the limiting process and convergence}

Let us identify the limiting process. By lemma~\ref{LemTension}, there exists a subsequence $\smash{\big(\Ss_{m_n}(t),\,t\geq 0\big)_{n\geq 1}}$ which converges in distribution to some process $(\Uc(t),t\geq 0)$ on $\Dc([0,T],\R)$. By lemma~\ref{LemArret}, $\smash{\big(\Ts_{m_n}(t),\,t\geq 0\big)_{n\geq 1}}$ converges in distribution to the null process on $\Dc([0,T],\R)$.\medskip

For $k>0$, we introduce the stopping time
\[\widetilde{\t}^k_n=\min\,\left(\,T\,,\,\inf_{t\geq 0}\,\Big\{\,\,\big|\Ts_n(t)\big|\geq k\,\Big\}\,\right).\]
If $t\geq T$ then $\P(\widetilde{\t}^k_n\leq t)=1$ and, if $t<T$, then
\[\lim_{n\to+\infty}\P(\widetilde{\t}^k_n\leq t)\leq \lim_{n\to+\infty}\P\left(\,\sup_{0\leq t\leq T}\,\big|\Ts_n(t)\big|\geq k\,\right)=0,\]
by lemma~\ref{LemArret}. As a consequence $\smash{(\widetilde{\t}^k_n)_{n\geq 1}}$ converges in distribution to $T$.\medskip

We give ourselves $f\in C^4(\R)$. For any $n\geq 1$ and $t\in[0,T]$,
\[F_{n,f}\big(\Ss_n(t),\Ts_n(t)\big)=f\big(\Ss_n(t)\big)+\left(\frac{1}{n^{1/4}}H_f+\frac{1}{n^{1/2}}K_f\right)\big(\Ss_n(t),\Ts_n(t)\big),\]
the functions $H_f$ and $K_f$ being continuous. Next, proposition~\ref{GenS_n()} implies that, for any $n\geq 1$ and $t\in [0,T]$,
\[\Gs_n F_{n,f}\big(\Ss_n(t),\Ts_n(t)\big)=G_{\s}f\big(\Ss_n(t)\big)+\smash{\widetilde{R}_{n,f}}\big(\Ss_n(t),\Ts_n(t)\big),\]
where $\smash{\widetilde{R}_{n,f}}$ is a continuous function on $\R^2$ such that
\[\forall k>0 \qquad \lim_{n\to+\infty}\, \sup_{\gfrac{(x,y) \in \R^2}{\|(x,y)\|\leq k}}\,\big|\smash{\widetilde{R}_{n,f}}(x,y)\big|=0.\]
Let $k>0$. For any $t\geq 0$, we obtain
\[\Mc_{m_n,f}(t\wedge \widetilde{\t}^k_n)\overset{\loi}{\underset{n\to +\infty}{\longrightarrow}}\Mc_f(t )=f(\Uc(t\wedge T))-f(\Uc(0))-\int_0^{t\wedge T} G_{\s}f(\Uc(s))\,ds.\]
For all $n\geq 1$ and $t\in [0,T]$, we have
\[\langle \Mc_{n,f}(\,\cdot\,\wedge \widetilde{\t}^k_n),\Mc_{n,f}(\,\cdot\,\wedge \widetilde{\t}^k_n)\rangle_t=\sqrt{n}\sum_{j=1}^n\int_0^{t\wedge \widetilde{\t}^k_n}\left(\frac{\partial \Psi_{n,f}}{\partial x_j}\right)^2\big(X_n(\sqrt{n}s)\big)\,ds,\]
and, using formula~(\ref{Psi(n,f)}), we get
\begin{multline*}
\sqrt{n}\sum_{j=1}^n\left(\frac{\partial \psi_{n,f}}{\partial x_j}\right)^2\big(X_n(\sqrt{n}\,\cdot\,)\big)\\[-0.3cm]
=\left(f'\big(\Ss_n\big)+\left[\frac{1}{n^{1/4}}\frac{\partial H_f}{\partial x}+\frac{1}{n^{1/2}}\frac{\partial K_f}{\partial x}\right]\big(\Ss_n,\Ts_n\big)\right)^2\\
\qquad\qquad+\left(\frac{4\Ts_n}{n^{1/4}}+4\s^2\right)\left(\left[\frac{1}{n^{1/4}}\frac{\partial H_f}{\partial y}+\frac{1}{n^{1/2}}\frac{\partial K_f}{\partial y}\right]\big(\Ss_n,\Ts_n\big)\right)^2\\
\qquad\qquad\qquad+\frac{4\Ss_n}{n^{1/4}}\left(f'\big(\Ss_n\big)+\left[\frac{1}{n^{1/4}}\frac{\partial H_f}{\partial x}+\frac{1}{n^{1/2}}\frac{\partial K_f}{\partial x}\right]\big(\Ss_n,\Ts_n\big)\right)\\
\qquad\qquad\qquad\qquad\qquad\times\left(\left[\frac{1}{n^{1/4}}\frac{\partial H_f}{\partial y}+\frac{1}{n^{1/2}}\frac{\partial K_f}{\partial y}\right]\big(\Ss_n,\Ts_n\big)\right).
\end{multline*}
Assume that $f$ has a compact support. Then we observe that there exists a constant $\widetilde{C}^k_f$ such that
\[\big|\Ts_n(t)\big|\leq k\qquad\Longrightarrow\qquad \sqrt{n}\sum_{j=1}^n\left(\frac{\partial \psi_{n,f}}{\partial x_j}\right)^2\big(X_n(\sqrt{n}t)\big)\leq \widetilde{C}^k_f.\]
As a consequence $\smash{\Mc_{n,f}(\,\cdot\,\wedge \widetilde{\t}^k_n)}$ is a martingale and
\[\forall t\geq 0\qquad\sup_{n\geq 1}\,\E\left(\Mc_{n,f}(t\wedge \widetilde{\t}^k_n)^2\right)\leq \widetilde{C}^k_f T <+\infty.\]
This implies that, for all $t\geq 0$, $\smash{\big(\Mc_{m_n,f}(t\wedge \widetilde{\t}^k_{m_n})\big)_{n\geq 1}}$ is an uniformly integrable family. Therefore $\Mc_f$ is a martingale.\medskip

Theorem~1.7 of chapter~8 of~\cite{EKurtz} implies that the martingale problem associated to $\{\,(f,G_{\s}f) : f\in C_c^{\infty}(\R)\,\}$ admits a unique solution: it is the strong solution of the differential stochastic equation
\[dz(t)=-\frac{z^3(t)}{2\s^4}\,dt+dB(t),\qquad z(0)=0,\]
where $(B(t),\,t\geq 0)$ is a standard Brownian motion. As a consequence the limiting process $(\Uc(t),\,0\leq t\leq T)$ is uniquely determined. Therefore
\[\displaystyle{\left(\frac{S_n(\sqrt{n}t)}{n^{3/4}},\,0\leq t\leq T\right)_{n\geq 1}=\big(\Ss_n(t),\,0\leq t\leq T\big)_{n\geq 1}}\]
converges in distribution to $(\Uc(t),\,0\leq t\leq T)$ on $\Dc([0,T],\R)$. Finally, since the sample paths of $(\Uc(t),\,0\leq t\leq T)$ are continuous, this convergence in distribution holds in $C([0,T],\R)$. This ends the proof of theorem~\ref{MainTheoGauss}.

\setcounter{section}{0}
\renewcommand{\thesection}{\!\!\!\!\!\!}
\setcounter{theo}{0}
\renewcommand{\thetheo}{\Alph{section}.\arabic{theo}}

\section{Appendix\\A proposition on collapsing processes}

\begin{defi} A sequence of real-valued stochastic processes $(\xi_n(t),\,t\geq 0)_{n\geq 1}$ collapses to zero if
\[\forall\eps>0\quad\forall T>0\qquad \lim_{n\to+\infty}\,\P\left(\,\sup_{0\leq t\leq T}\,\left|\xi_n(t)\right|>\eps\,\right)=0.\]
\end{defi}

The concept of collapsing processes has been developed by Francis Comets and Theodor Eisele in~\cite{ComEis}.

\begin{prop} Let $(\xi_n(t),\,t\geq 0)_{n \geq 1}$ be a sequence of positive semimartingales on a probability space $(\Omega,\Fc, \P)$. For any $n\geq 1$, we give ourselves an integer $m_n\geq 1$ and independent standard Brownian motions $\left(B_i \right)_{1\leq i\leq m_n}$ which generate a filtration $\left(\Fc_t \right)_{t \geq 0}$. We assume that there exist $\left(\Fc_t \right)_{t \geq 0}$-adapted processes $(\z_n(t),\,t\geq 0)$ and $(Z_{n,i}(t),\,t\geq 0)_{1\leq i\leq m_n}$ such that
\[
d\xi_n(t) = \z_n(t)dt + \sum_{i=1}^{m_n} Z_{n,i}(t) dB_i(t).
\]
We suppose that there exist $d>1$, positive constants $C_1,\dots,C_5$, increasing sequences $(\kappa_n)_{n \geq1}$, $(\alpha_n)_{n \geq1}$, $(\beta_n)_{n \geq 1}$ and a sequence $(\tau_n)_{n \geq 1}$ of stopping times verifying 
\begin{equation} \tag{$\mathscr{C}_1$}\label{C1}
\kappa_n^{\frac{1}{d}}\alpha_n^{-1} \underset{n\to+\infty}{\longrightarrow} 0, \qquad \kappa_n^{-1} \alpha_n  \underset{n\to+\infty}{\longrightarrow} 0, \qquad \kappa_n^{-1}\beta_n  \underset{n\to+\infty}{\longrightarrow} 0,
\end{equation}
\begin{equation}\tag{$\mathscr{C}_2$}\label{C2}
\forall n\geq 1\qquad \E \left[ \Big( \xi_n(0) \Big)^d \right] \leq C_1 \alpha_n^{-d}, 
\end{equation}
\begin{equation} \tag{$\mathscr{C}_3$}\label{C3}
\forall n\geq 1\quad\forall t\in [0,\tau_n]\qquad \z_n(t) \leq -\kappa_n C_2 \xi_n(t) + \beta_n C_3 + C_4,
\end{equation}
and
\begin{equation}\tag{$\mathscr{C}_4$}\label{C4}
\forall n\geq 1\quad\forall t\in [0,\tau_n]\qquad \sum_{i=1}^{m_n} Z_{n,i}(t)^2 \leq C_5 .
\end{equation}
Then, for any $\eps>0$ and $T>0$, there exist $M>0$ and $n_0\geq 1$ such that
\[\sup_{n \geq n_0} \,\P \left(\, \sup_{0 \leq t \leq T \wedge \tau_n} \xi_n (t) > M \left( \kappa_n^{\frac{1}{d}} \alpha_n^{-1} \vee \alpha_n \kappa_n^{-1} \right)\,\right) \leq \eps.\]
\label{CollapsingProp}
\end{prop}

This is proposition~4.2~of~\cite{DaiPra}. It is a simple adaptation of the proposition in appendix~of~\cite{ComEis}.

\bibliographystyle{plain}
\bibliography{biblio}

\end{document}